\documentclass[reqno]{amsart}
\usepackage{amsmath, amssymb, amsthm, epsfig}
\usepackage{hyperref, latexsym}
\usepackage{url}
\usepackage[mathscr]{euscript}
\usepackage{amsthm} 
\usepackage{bbm}

\usepackage{color}
\usepackage{fullpage} 
\usepackage{setspace}

\usepackage{array,booktabs}

\def\today{\ifcase\month\or
  January\or February\or March\or April\or May\or June\or 
  July\or August\or September\or October\or November\or December\fi
  \space\number\day, \number\year}

 \newtheorem{theorem}{Theorem}[section]

 \newcommand{\mc}{\mathcal}

 \newcommand{\M}{\mc{M}}

 \newcommand{\R}{\mathbb{R}}
  
 \newcommand{\N}{\mathbb{N}}
 
 \newcommand{\Z}{\mathbb{Z}}

 \newcommand{\ds}{\text{\rm d}s}

 \newcommand{\dx}{\text{\rm d}x}
 
 \newcommand{\dy}{\text{\rm d}y}

    \renewcommand{\d}{\text{\rm d}}

\newcommand{\wt}{\widetilde}

\newcommand{\var}{{\rm Var\,}}

\newcommand{\one}{\mathbbm 1}

\begin{document}

\title[]{Regularity of maximal operators: \\
recent progress and some open problems}
\author[Carneiro]{Emanuel Carneiro}
\date{\today}

\address{
ICTP - The Abdus Salam International Centre for Theoretical Physics\\
Strada Costiera, 11, I - 34151, Trieste, Italy.}
\address{IMPA - Instituto Nacional de Matem\'{a}tica Pura e Aplicada - Estrada Dona Castorina, 110, Rio de Janeiro, RJ 22460-320, Brazil.}
\email{carneiro@ictp.it}
\email{carneiro@impa.br}

\allowdisplaybreaks
\numberwithin{equation}{section}

\maketitle

\begin{abstract}
This is an expository paper on the regularity theory of maximal operators, when these act on Sobolev and BV functions, with a special focus on some of the current open problems in the topic. Overall, a list of fifteen research problems is presented. It summarizes the contents of a talk delivered by the author in the CIMPA 2017 Research School - Harmonic Analysis, Geometric Measure Theory and Applications, in Buenos Aires, Argentina.
\end{abstract}

\section{Introduction} 
Maximal operators are classical objects in analysis. They usually arise as important tools to prove different sorts of pointwise convergence results, e.g. Lebesgue's differentiation theorem, Carleson's theorem on the pointwise convergence of Fourier series, pointwise convergence of solutions of PDEs to the initial datum, and so on. Despite being extensively studied for decades, maximal operators still conceal some of their secrets, and understanding the intrinsic mapping properties of these operators in different function spaces still remains an active topic of research.

\smallskip

Throughout this paper we focus on the most classical of these objects, the Hardy-Littlewood maximal operator, and some of its variants. As we shall see, it will be important for our discussion to consider the centered and uncentered versions of this operator, discrete analogues, fractional analogues and convolution-type analogues. For $f \in L^1_{loc}(\R^d)$ we define the {\it centered Hardy-Littlewood maximal function} $Mf$ by 
\begin{equation}\label{HL_Max}
Mf(x) = \sup_{r > 0} \frac{1}{m(B(x,r))} \int_{B(x,r)} |f(y)|\,\dy\,,
\end{equation}
where $B(x,r)$ is the open ball of center $x$ and radius $r$, and $m(B(x,r))$ denotes its $d$-dimensional Lebesgue measure. The {\it uncentered maximal function} $\wt{M}f$ at a point $x$ is defined analogously, taking the supremum of averages over open balls that contain the point $x$, {\it but that are not necessarily centered at $x$.}

\smallskip

One of the fundamental results in harmonic analysis is the theorem of Hardy and Littlewood that states that $M: L^1(\R^d) \to L^{1,\infty}(\R^d)$ is a bounded operator. By interpolation with the trivial $L^{\infty}$-estimate, this yields the boundedness of $M: L^p(\R^d) \to L^{p}(\R^d)$ for $1 < p \leq \infty$. Another consequence of the weak-$(1,1)$ bound for $M$ is the Lebesgue's differentiation theorem. The $L^p$-mapping properties of the uncentered maximal operator $\wt{M}$ are exactly the same.

\smallskip

One may consider the action of the Hardy-Littlewood maximal operator in other function spaces and investigate whether it improves, preserves or destroys the a priori regularity of an initial datum $f$. This type of question is essentially the main driver of what we refer to here as {\it regularity theory for maximal operators}. Let us denote by $W^{1,p}(\R^d)$ the Sobolev space of functions $f \in L^p(\R^d)$ that have a weak gradient $\nabla f \in L^p(\R^d)$, with norm given by
$$\|f\|_{W^{1,p}(\R^d)} = \|f\|_{L^p(\R^d)} + \|\nabla f\|_{L^p(\R^d)}.$$
In 1997, J. Kinnunen wrote an enlightening paper \cite{Ki}, establishing the boundedness of the operator $M: W^{1,p}(\R^d) \to W^{1,p}(\R^d)$ for $1 < p \leq \infty$. This marks the beginning of our story. After that, a number of interesting works have devoted their attention to the investigation of the action of maximal operators on Sobolev spaces and on the closely related space of functions of bounded variation. This survey paper is brief account of some of the developments in this topic over the last 20 years, with a special focus on a list of fifteen open problems that may guide new endeavors.

\smallskip

The choice of topics and problems presented here is obviously biased by the personal preferences of the author and is by no means exhaustive. We shall present just a couple of brief proofs of some of the earlier results to give a flavor to the reader of what is going on, for the main purpose of this expository paper is to provide a light and inviting reading on the topic, especially to newcomers. We refer the reader to the original papers for the proofs of the results mentioned here.

\smallskip

For simplicity, all functions considered in this paper are real-valued functions.

\section{Kinnunen's seminal work}
Let us start by revisiting the main result of \cite{Ki} and its elegant proof.

\begin{theorem}[Kinnunen, 1997 - cf. \cite{Ki}]\label{Thm_Kin}
Let $1 < p < \infty$ and let $f \in W^{1,p}(\R^d)$. Then $Mf$ is weakly differentiable and
\begin{equation}\label{upper bound_0}
\big|\partial_i Mf(x)\big| \leq M\big(\partial_if)(x)
\end{equation}
for almost every $x \in \R^d$. Therefore $M: W^{1,p}(\R^d) \to W^{1,p}(\R^d)$ is bounded.
\end{theorem}
\begin{proof} 
Let $B = B(0,1) \subset \R^d$ be the unit ball and define
\begin{equation}\label{def_varphi_ball}
\varphi(x) := \frac{\one_B(x)}{m(B)},
\end{equation}
where $\one_B$ is the characteristic function of $B$. For $r >0$ let us define 
\begin{equation*}
\varphi_r(x) := r^{-d}\,\varphi(x/r).
\end{equation*}
With this notation we plainly have
$$Mf(x) = \sup_{r>0}\,(\varphi_r*|f|)(x).$$
Fix $1 \leq i \leq d$. Recall that if $f \in W^{1,p}(\R^d)$ then $|f| \in W^{1,p}(\R^d)$ and $|\partial_i |f|| = |\partial_if|$ almost everywhere (see for instance \cite[Theorem 6.17]{LL}). Let us enumerate the positive rational numbers as $\{r_1, r_2, r_3, \ldots\}$ and define $h_j := \varphi_{r_j}*|f|$. Then $h_j \in W^{1,p}(\R^d)$ and $\partial_i h_j = \varphi_{r_j}*\partial_i|f|$. 

\smallskip

Let $N\geq 1$ be a natural number and define $g_N(x) := \max_{1 \leq j \leq N} h_j(x)$. Note that $g_N \in W^{1,p}(\R^d)$ with 
$$g_N(x) \leq Mf(x)$$
and (see \cite[Theorem 6.18]{LL})
\begin{equation}\label{upper bound_1}
|\partial_i g_N(x)|  \leq \max_{1 \leq j \leq N} |\partial_ih_j(x)| \leq M(\partial_i f)(x)
\end{equation}
for almost every $x \in \R^d$. Then $\{g_N\}_{N\geq 1}$ is a bounded sequence in $W^{1,p}(\R^d)$ with the property that $g_N(x) \to Mf(x)$ pointwise as $N \to \infty$. Since $W^{1,p}(\R^d)$ is a reflexive Banach space, by passing to a subsequence if necessary, we may assume that $g_N$ converges weakly to a function $g \in W^{1,p}(\R^d)$ (a crucial point in this argument is that this weak limit is already born in $W^{1,p}(\R^d)$). Standard functional analysis tools (for instance, using Mazur's lemma \cite[Corollary 3.8 and Exercise 3.4]{Br}) lead to the conclusion that $Mf = g \in W^{1,p}(\R^d)$ and that the upper bound \eqref{upper bound_1} is preserved almost everywhere up to the weak limit. The latter assertion leads to \eqref{upper bound_0}.
\end{proof}
We call the attention of the reader for the use of the reflexivity of the space $W^{1,p}(\R^d)$, for $1 < p < \infty$, in the conclusion of the proof above. This is one of the obstacles when one considers the endpoint case $p=1$, as we shall see in the next section. The case $p=\infty$ can be dealt with directly. In fact, if $f \in W^{1,\infty}(\R^d)$ then $f$ can be modified on a set of measure zero to become Lipschitz continuous, with Lipschitz constant $L \leq \|\nabla f\|_{\infty}$. With the notation of the proof above, for a fixed $r$, each average $\varphi_r*|f|$ is Lipschitz with constant at most $L$. The pointwise supremum of uniformly Lipschitz functions is still Lipschitz with (at most) the same constant. This shows that $M: W^{1,\infty}(\R^d) \to W^{1,\infty}(\R^d)$ is a bounded operator.

\smallskip

If one is not necessarily interested in pointwise estimates for the derivative of the maximal function, there is a simpler argument using the characterization of Sobolev spaces via difference quotients \cite[Theorem 1]{HO}. This covers the general situation of sublinear operators that commute with translations. Recall that an operator $A:X \to Y$, acting between linear function spaces $X$ and $Y$, is said to be sublinear if $Af \geq 0$ a.e. for $f \in X$, and $A(f+g) \leq Af + Ag$ a.e. for $f,g \in X$. In what follows we let $f_y(x) :=f(x+ y)$ for $x, y \in \R^d$.

\begin{theorem}[Haj\l asz and Onninen, 2004 - cf. \cite{HO}]\label{A_operator}
Assume that the operator $A: L^p(\R^d) \to L^p(\R^d)$, $1 < p < \infty$, is bounded and sublinear. If $A(f_y) = (Af)_y$ for all $f \in L^p(\R^d)$ and all $y \in \R^d$, then $A: W^{1,p}(\R^d) \to W^{1,p}(\R^d)$ is also bounded.
\end{theorem}
\begin{proof} Let $e_i$ be the unit coordinate vector in the $x_i$ direction. For $t >0$ we have
\begin{align*}
\|(Af)_{te_i} - Af\|_{L^p(\R^d)} &= \|A(f_{te_i}) - Af\|_{L^p(\R^d)} \\
& \leq \|A(f_{te_i} - f)\|_{L^p(\R^d)} + \|A(f - f_{te_i})\|_{L^p(\R^d)} \\
& \leq C \,\|f_{te_i} - f\|_{L^p(\R^d)}\\
& \leq C \,t\,\|\partial_i f\|_{L^p(\R^d)}.
\end{align*}
The last inequality above follows from \cite[Lemma 7.23]{GT}. Since the difference quotients $\big\{((Af)_{te_i} - Af)/t\big\}_{t>0}$ are uniformly bounded in $L^p(\R^d)$, an application of \cite[Lemma 7.24]{GT} guarantees that $Af \in W^{1,p}(\R^d)$ and 
$$\|\partial_i Af \|_{L^p(\R^d)} \leq C \,\|\partial_i f\|_{L^p(\R^d)}.$$
This concludes the proof.
\end{proof}
In the scope of Theorem \ref{A_operator}, one may consider the spherical maximal operator. Letting $S^{d-1}(x,r) \subset \R^d$ be the $(d-1)$-dimensional sphere of center $x$ and radius $r$, this operator is defined as
\begin{equation}\label{Spherical_max}
M_{S}f(x) = \sup_{r>0} \, \frac{1}{\omega_{d-1}r^{d-1}}\int_{S^{n-1}(x,r)} |f(z)|\,\d\sigma(z),
\end{equation}
where $\sigma$ is the canonical surface measure on $S^{d-1}(x,r)$ and $\omega_{d-1} = \sigma\big(S^{d-1}(0,1)\big)$. A remarkable result of Stein \cite{Stein1976} in dimension $d \geq 3$, and Bourgain \cite{Bourgain} in dimension $d=2$, establishes that $M_{S}: L^p(\R^d) \to L^p(\R^d)$ is a bounded operator for $p > d/(d-1)$. It plainly follows from Theorem \ref{A_operator} that $M_{S}: W^{1,p}(\R^d) \to W^{1,p}(\R^d)$ is also bounded for $p > d/(d-1)$ (the case $p = \infty$ is treated directly). 

\smallskip 

Theorem \ref{Thm_Kin} has been extended in many different ways over the last years, and we now mention a few of such related results. Kinnunen and Lindqvist \cite{KL} extended Theorem \ref{Thm_Kin} to a local version of the maximal operator. In this setting one considers a domain $\Omega \subset \R^d$, functions $f \in W^{1,p}(\Omega)$, and the maximal operator is taken over balls entirely contained in the domain $\Omega$. Extensions of Theorem \ref{Thm_Kin} to a multilinear setting are considered in the work of the author and Moreira \cite{CM} and by Liu and Wu \cite{LW}, and a similar result in fractional Sobolev spaces is the subject of the work of Korry \cite{Ko}. A fractional version of the Hardy-Littlewood maximal operator is considered in the paper \cite{KiSa} by Kinnunen and Saksman (we will return to this particular operator later on). An interesting variant of this result on Hardy-Sobolev spaces is considered in the recent work of P\'{e}rez, Picon, Saari and Sousa \cite{PPSS}.

\section{The endpoint Sobolev space}
With the philosophy that averaging is a smoothing process, we  would like to understand if certain smoothing features are still preserved when we take a pointwise supremum over averages. Understanding the situation described in Theorem \ref{Thm_Kin} at the endpoint case $p=1$ is a subtle issue. Of course, if $f \in L^1(\R^d)$ is non-identically zero, we already know that $f \notin L^{1}(\R^d)$, and the interesting question is whether one can control the behavior of the derivative of the maximal function. The following question was raised in the work of Haj\l asz and Onninen \cite[Question 1]{HO} and remains one of the main open problems in the subject.

\medskip

\noindent {\bf Question 1} (Haj\l asz and Onninen, 2004 - cf. \cite{HO}). {\it Is the operator $f \mapsto |\nabla Mf|$ bounded from $W^{1,1}(\R^d)$ to $L^1(\R^d)$\,{\rm?} Same question for the uncentered operator $\widetilde{M}$.}

\medskip

\noindent Naturally, this involves proving that $Mf$ is weakly differentiable, and establishing the bound 
\begin{equation}\label{Sec3_Bound_1}
\|\nabla Mf\|_{L^1(\R^d)} \leq C \big(\|f\|_{L^1(\R^d)} + \|\nabla f\|_{L^1(\R^d)}\big),
\end{equation}
for some universal constant $C = C(d)$. If the global estimate \eqref{Sec3_Bound_1} holds for every $f \in W^{1,1}(\R^d)$, a simple dilation argument implies that one should actually have
\begin{equation*}
\|\nabla Mf\|_{L^1(\R^d)} \leq C \,\|\nabla f\|_{L^1(\R^d)},
\end{equation*}
which reveals the true nature of the problem: if one can control the variation of the maximal function by the variation of the original function (the term {\it variation} here is used as the $L^1$-norm of the gradient).

\smallskip

Several interesting papers addressed Question $1$, which has been answered affirmatively in dimension $d=1$, but remains vastly open in dimensions $d\geq2$. We now comment a bit on these results.

\subsection{One-dimensional results} The achievements in dimension $d=1$ started with the work of Tanaka \cite{Ta}, for the uncentered maximal operator $\widetilde{M}$. In this particular work, Tanaka showed that if $f \in W^{1,1}(\R)$ then $\widetilde{M} f$ is weakly differentiable and 
\begin{equation}\label{Tanaka_boy}
\big\|\big(\widetilde{M} f\big)'\big\|_{L^1(\R)} \leq 2 \,\|f'\|_{L^1(\R)}
\end{equation}
(see also \cite{LCW}). This result was later refined by Aldaz and P\'{e}rez L\'{a}zaro in \cite[Theorem 2.5]{AP}. Letting $\var(f)$ denote the total variation of a function $f:\R \to \R$, they proved the following very interesting result.
\begin{theorem}[Aldaz and P\'{e}rez L\'{a}zaro, 2007 - cf. \cite{AP}]\label{Thm_APL}
Let $f:\R \to \R$ be a function of bounded variation. Then $\widetilde{M} f$ is an absolutely continuous function and we have the inequality
\begin{equation}\label{Var_Sharp}
\var\big( \widetilde{M} f \big) \leq \var(f).
\end{equation}
\end{theorem}
We comment on the two main features of this theorem. Firstly, the regularizing effect of the operator $\widetilde{M}$, that takes a mere function of bounded variation into an absolutely continuous function. The proof of this fact relies on the classical Banach-Zarecki theorem. This regularizing effect is not shared by the centered maximal operator $M$, as it can be seen by simply taking $f$ to be the characteristic function of an interval. In this sense, the uncentered operator is more regular than the centered one, and in many instances in this theory it is a more tractable object. Secondly, the inequality \eqref{Var_Sharp} with constant $C=1$ is sharp, as it can be seen again by taking $f$ to be the characteristic function of an interval. Note, in particular, that \eqref{Var_Sharp} indeed refines \eqref{Tanaka_boy}, since any function $f \in W^{1,1}(\R)$ can be modified on a set of measure zero to become absolutely continuous. The core of this argument comes from the fact that the maximal function does not have points of local maxima in the set where it disconnects from the original function (sometimes referred to here as {\it detachment set}).

\smallskip

Proving an inequality of the same spirit as \eqref{Var_Sharp} for the one-dimensional centered Hardy-Littlewood maximal operator is a harder task. In this situation, there may be local maxima of $Mf$ in the detachment set (one may see this for instance by considering $f = \delta_{-1} + \delta_{1}$, where $\delta_{x_0}$ denotes the Dirac delta function at the point $x_0$; in this case the point $x=0$ is a local maximum for $Mf$; of course, technically one would have to smooth out this example to view the Dirac deltas as actual functions) and the previous argument of Aldaz and P\'{e}rez L\'{a}zaro cannot directly be adapted. In the work \cite{Ku}, O. Kurka proved the following remarkable result.
\begin{theorem}[Kurka, 2015 - cf. \cite{Ku}]\label{Thm_Kurka}
Let $f:\R \to \R$ be a function of bounded variation. Then 
\begin{equation}\label{Var_Sharp_centered}
\var(M f) \leq 240004 \, \var(f).
\end{equation}
\end{theorem}
\noindent The proof of this theorem relies on a beautiful, yet rather intricate, argument of induction on scales (from which one arrives at the particular constant $C = 240004$). Things seem to be tailor-made to the case of the Hardy-Littlewood maximal function, and it would be interesting to see if the argument can be adapted to treat other convolution kernels (discussed in the next section). The constant $C = 240004$ is certainly intriguing, but there seems to be no philosophical reason to justify this order of magnitude. This leaves the natural open question.

\medskip

\noindent {\bf Question 2.} {\it Let $f:\R \to \R$ be a function of bounded variation. Do we have
$$\var(M f) \leq  \var(f)\,?$$
Or at least, can one substantially improve on Kurka's constant $C = 240004$}\,?

\medskip

Despite the innocence of the statement of Question 2, the reader should not underestimate its difficulty. As a matter of fact, the reader is invited to think a little bit about this question to get acquainted with some of its obstacles. This is a beautiful example of an open question in this research topic. These usually have relatively simple statements and their solutions might only require ``elementary" tools, but the difficulty lies in {\it how to properly combine these tools.}

\smallskip

Recently, J. P. Ramos \cite{Ra} considered a hybrid version between $M$ and $\widetilde{M}$ in dimension $d=1$. For $\alpha \geq 0$, we may define the non-tangential maximal operator $M^{\alpha}$ by 
$$M^{\alpha}f(x) := \sup_{(y,t)\,:\, |x-y|\leq \alpha t}\, \frac{1}{2t} \int_{y-t}^{y+t} |f(s)|\,\ds.$$
In this setting, we notice that $M^{0} = M$ and $M^1 = \widetilde{M}$. Ramos shows that 
\begin{equation}\label{hier}
\var(M^{\alpha}f) \leq \var(M^{\beta}f)
\end{equation}
if $\alpha \geq \beta$, and from Theorem \ref{Thm_Kurka} one readily sees that 
$$\var(M^{\alpha}f) \leq C\, \var(f)$$
for all $\alpha \geq 0$. From Theorem \ref{Thm_APL} we may take $C=1$ if $\alpha \geq 1$. Ramos \cite[Theorem 1]{Ra} goes further and establishes the following result.

\begin{theorem}[Ramos, 2017 - cf. \cite{Ra}]\label{Thm_Ramos}
Let $\alpha \in [\frac{1}{3}, \infty)$ and let $f:\R \to \R$ be a function of bounded variation. Then 
\begin{equation}\label{Ineq_Ramos}
\var(M^{\alpha} f) \leq  \var(f).
\end{equation}
\end{theorem}
\noindent The constant $C=1$ in inequality \eqref{Ineq_Ramos} is sharp as it can be easily seen by taking $f$ to be the characteristic function of an interval. The proof of Ramos for Theorem \ref{Thm_Ramos} extends the argument of Aldaz and P\'{e}rez L\'{a}zaro \cite{AP}, in particular establishing the crucial property that $M^{\alpha} f$ has no local maxima in the detachment set for $\alpha >1/3$. The case $\alpha =1/3$ in \eqref{Ineq_Ramos} is obtained by a limiting argument. The interesting thing here is that $\alpha =1/3$ is the threshold for this property. Indeed, if $\alpha <1/3$, by taking $f =\delta_{-1} + \delta_1$ we see that $x=0$ is a local maximum of $M^{\alpha} f$, see \cite[Theorem 2]{Ra} (again, one must smooth out this example, since the Dirac deltas are actually singular measures and not exactly functions of bounded variation -- but this can be done with no harm). We conclude the one-dimensional discussion with the following question (which, by \eqref{hier}, would follow from an affirmative answer to Question 2).

\medskip

\noindent {\bf Question 3.} {\it Let $f:\R \to \R$ be a function of bounded variation. Do we have
\begin{equation}\label{Eq_alpha_var}
\var(M^{\alpha} f) \leq  \var(f)
\end{equation}
for $0 \leq \alpha < \frac{1}{3}$\,{\rm ?} Alternatively, what is the smallest value of $\alpha$ for which \eqref{Eq_alpha_var} holds}\,?

\subsection{Multidimensional results} Question 1 remains open, in general, for dimensions $d \geq 2$. There have been a few particular works that made interesting partial progress and we now comment on three of them, namely \cite{HM, Lu3, Saa}.

\smallskip

In the paper \cite{HM}, Haj\l asz and Mal\'{y} consider a slightly weaker notion of differentiability. A function $f:\R^d \to \R$ is said to be  {\it approximately differentiable} at the point $x_0 \in \R^d$ if there exists a vector $L = (L_1, L_2, \ldots, L_d)$ such that for any $\varepsilon >0$ the set 
\begin{equation*}
A_{\varepsilon} := \left\{ x \in \R^d\,:\, \frac{|f(x) - f(x_0) - L(x-x_0)|}{|x-x_0|} < \varepsilon \right\}
\end{equation*}
has $x_0$ as a density point. If this is the case, the vector $L$ is unique determined and it is called the {\it approximate differential} of $f$ at $x_0$. This is a weaker notion than that of classical differentiability or weak differentiability. In fact, if a function $f$ is differentiable at a point $x_0$ then it is approximately differentiable at $x_0$ and $L = \nabla f(x_0)$, and similarly, if $f$ is weakly differentiable then it is approximately differentiable and its approximate differential is equal to the weak derivative a.e., see for instance \cite[Section 6.1.3, Theorem 4]{EG}. The main result of \cite{HM} reads as follows.
\begin{theorem}[Haj\l asz and Mal\'{y}, 2010 - cf. \cite{HM}]
If $f \in L^1(\R^d)$ is approximately differentiable a.e. then the maximal function $Mf$ is approximately differentiable a.e.
\end{theorem}
The recent interesting work of Luiro \cite{Lu3} answers Question 1 affirmatively in the case of the uncentered maximal function $\widetilde M$ and {\it restricted to radial functions $f$}. 
\begin{theorem}[Luiro, 2017 - cf. \cite{Lu3}]\label{Luiro2017}
If $f \in W^{1,1}(\R^d)$ is radial, then $\widetilde{M}f$ is weakly differentiable and 
\begin{equation*}
\big\|\nabla \widetilde{M}f\big\|_{L^1(\R^d)} \leq C\, \|\nabla f\|_{L^1(\R^d)},
\end{equation*}
where $C = C(d)$ is a universal constant.
\end{theorem}
\noindent This raises a natural question, another interesting particular case of Question 1.

\medskip

\noindent {\bf Question 4.} {\it Does Theorem \ref{Luiro2017} hold for the centered maximal operator $M$ acting on radial functions $f$}\,?

\medskip

In \cite{Saa}, Saari studies the regularity of maximal operators via generalized Poincar\'{e} inequalities. An interesting corollary \cite[Corollary 4.1]{Saa} of the main result of this paper establishes that, if $f \in W^{1,1}(\R^d)$, then the distributional partial derivatives $\partial_i\widetilde{M}f$ (or $\partial_i Mf$)  can be represented as functions $h_i \in L^{1,\infty}(\R^d)$ when they act on smooth functions with compact support not meeting a certain singularity set.
 
\smallskip
 
Finally, let us briefly return to the spherical maximal operator $M_S$ defined in \eqref{Spherical_max}. Recall that we have shown in the previous section that $M_S$ is bounded in $W^{1,p}(\R^d)$ for $d \geq 2$ and $p > d/(d-1)$. We conclude this section with the following question, originally proposed in \cite[Question 2]{HO}.

\medskip

\noindent {\bf Question 5.} (Haj\l asz and Onninen, 2004 - cf. \cite{HO}) {\it Let $d\geq 2$. Is $M_S: W^{1,p}(\R^d) \to W^{1,p}(\R^d)$ bounded for $1 < p \leq d/(d-1)$}?

\medskip

\noindent Note that $p=1$ is not actually part of Question 5. In this case, the operator $f \mapsto |\nabla M_Sf|$ is not bounded from $W^{1,1}(\R^d)$ to $L^1(\R^d)$. A counterexample is given in \cite{HO}.

\section{Maximal operators of convolution type} 
Let $\varphi:\R^d \to \R$ be a nonnegative and integrable function with
$$\int_{\R^d} \varphi(x)\,\dx = 1.$$ 
As before, for $t >0$, we let $\varphi_t = t^{-d}\,\varphi(x/t)$. We define here the {\it maximal operator of convolution type} associated to $\varphi$ by
\begin{equation*}
M_{\varphi}f(x) = \sup_{t>0} \big(\varphi_{t}*|f|\big)(x).
\end{equation*}
Recall that the centered Hardy-Littlewood maximal operator arises when the kernel $\varphi$ is given by \eqref{def_varphi_ball}. When $\varphi$ admits a radial decreasing majorant in $L^1(\R^d)$, a classical result of Stein \cite[Chapter III, Theorem 2]{S} establishes that $M_{\varphi}: L^{p}(\R^d) \to L^{p}(\R^d)$ is a bounded operator for $1 < p \leq \infty$, and at $p=1$ we have a weak-$(1,1)$ estimate. Theorem \ref{A_operator} plainly implies that $M_{\varphi}: W^{1,p}(\R^d) \to W^{1,p}(\R^d)$ is bounded for $p>1$ (again, the case $p=\infty$ can be dealt with directly), and we may ask ourselves the same sort of questions as in the previous section, with respect to the regularity of this operator at the endpoint Sobolev space $W^{1,1}(\R^d)$. 

\smallskip

As it turns out, we may have some advantages in considering certain smooth kernels. This additional leverage may come, for instance, from partial differential equations naturally associated to the kernel $\varphi$. This is well exemplified in the work \cite{CS}, where two special kernels are considered: the Poisson kernel
\begin{equation}\label{Intro_Poisson}
\varphi(x) = \frac{\Gamma \left(\frac{d+1}{2}\right)}{\pi^{(d+1)/2}}\ \frac{1}{(|x|^2 + 1)^{(d+1)/2}}\,,
\end{equation}
and the Gauss kernel
\begin{equation}\label{Intro_Gauss}
\varphi(x)  = \frac{1}{(4 \pi)^{d/2}}\ e^{-|x|^2/4}.
\end{equation}
For the Poisson kernel \eqref{Intro_Poisson}, the function $u(x,t) = \varphi_t(x)$ solves Laplace's equation $u_{tt} + \Delta_xu =0$ on the upper half-space $(x,t) \in \R^d \times (0,\infty)$. For the Gauss kernel \eqref{Intro_Gauss}, the function $u(x,t) = \varphi_{\sqrt{t}}(x)$ solves the heat equation $u_{t} - \Delta_xu =0$ on the upper half-space $(x,t) \in \R^d \times (0,\infty)$. The qualitative properties of these two partial differential equations (namely the corresponding maximum principles and the semigroup property) can be used to establish a positive answer for the convolution type analogue of Question 2 in these cases \cite[Theorems 1 and 2]{CS}.
\begin{theorem}[Carneiro and Svaiter, 2013 - cf. \cite{CS}]\label{Carn_Sveiter}
Let $\varphi$ be given by \eqref{Intro_Poisson} or \eqref{Intro_Gauss}, and let $f:\R \to \R$ be a function of bounded variation. Then
\begin{equation*}
\var(M_{\varphi}f) \leq  \var(f).
\end{equation*}
\end{theorem}

\noindent {\sc Remark}: In \cite[Theorems 1 and 2]{CS} it is also proved that, for every dimension $d \geq 1$, if $f \in W^{1,2}(\R^d)$, then we have (for $\varphi$ given by \eqref{Intro_Poisson} or \eqref{Intro_Gauss})
\begin{equation}\label{CS_general}
\big\|\nabla M_{\varphi}f\big\|_{L^2(\R^d)} \leq \|\nabla f\|_{L^2(\R^d)}.
\end{equation}
Additionally, if $d=1$, an analogous inequality to \eqref{CS_general} holds on $L^p(\R)$ for all $p > 1$.

\smallskip

Theorem \ref{Carn_Sveiter} has been extended to a larger family of kernels in the work \cite{CFS}. The general version of the result in Theorem \ref{Carn_Sveiter} is the theme of the following question.

\medskip

\noindent {\bf Question 6.} Let $\varphi$ be a convolution kernel as described above, and let $f:\R \to \R$ be a function of bounded variation. Can we show that 
\begin{equation}\label{Conv_Ker_Q}
\var(M_{\varphi}f) \leq  C\, \var(f)
\end{equation}
with $C = C(\varphi)$\,? For which $\varphi$ can we actually show \eqref{Conv_Ker_Q} with $C=1$?

\section{Fractional maximal operators}\label{Sec5}
For $0 \leq \beta < d$ and $f \in L^1_{loc}(\R^d)$ we define the centered Hardy-Littlewood fractional maximal function $M_{\beta}f$ by
\begin{equation*}
M_{\beta}f(x) = \sup_{r >0} \frac{1}{m(B(x,r))^{1 - \frac{\beta}{d}}} \int_{B(x,r)} |f (y)|\,\dy.
\end{equation*}
When $\beta =0$ we plainly recover \eqref{HL_Max}. The uncentered fractional maximal function $\widetilde{M}_{\beta}f$ is defined analogously, with the supremum of the fractional averages being taken over balls that simply contain the point $x$, but are not necessarily centered at $x$. Such fractional maximal operators have connections to potential theory and partial differential equations. By comparison with an appropriate Riesz potential, one can show that if $1 < p < \infty$, $0 < \beta < d/p$ and $q = dp/(d-\beta p)$, then $M_{\beta}: L^p(\R^d) \to L^q(\R^d)$ is bounded. When $p=1$ we have again a weak-type bound (for details, see \cite[Chapter V, Theorem 1]{S}).

\smallskip

In \cite{KiSa}, Kinnunen and Saksman studied the regularity properties of such fractional maximal operators, arriving at the following interesting conclusions \cite[Theorems 2.1 and 3.1]{KiSa}.
\begin{theorem}[Kinnunen and Saksman, 2003 - cf. \cite{KiSa}]\label{KS_2003}
Let $1 < p <\infty$. 
\begin{itemize}
\item[(i)] For $0 \leq \beta < d/p$ and $q = dp/(d-\beta p)$ the operator $M_{\beta}: W^{1,p}(\R^d) \to W^{1,q}(\R^d)$ is bounded.
\item[(ii)] Assume that $1 \leq \beta < d/p$ and that $f \in L^p(\R^d)$. Then $M_{\beta}f$ is weakly differentiable and there exists a constant $C  = C(d,\beta)$ such that 
$$|\nabla M_{\beta}f(x)| \leq C \, M_{\beta-1}f(x)$$
for almost every $x \in \R^d$.
\end{itemize}
\end{theorem}

\noindent Part (i) of Theorem \ref{KS_2003} extends the original result of Kinnunen (Theorem \ref{Thm_Kin}) to this fractional setting. One can prove it by using the characterization of the Sobolev spaces via the difference quotients as in the proof of Theorem \ref{A_operator}. Part (ii) of Theorem \ref{KS_2003} presents a beautiful regularization effect of this operator when the fractional parameter $\beta$ is greater than or equal to $1$.

\smallskip

In light of Theorem \ref{KS_2003}, it is then natural to ask ourselves what happens in the endpoint situation $p=1$ and $q = d/(d-\beta)$. Let us first consider the case $1 \leq \beta < d$. If $f \in W^{1,1}(\R^d)$, by the Sobolev embedding we have $f \in L^{p^*}(\R^d)$, where $p^* = d/(d-1)$, and hence $f \in L^{r}(\R^d)$ for any $1 \leq r \leq p^*$. We may choose $r$ with $1< r <d$ such that $1\leq \beta < d/r$. Using part (ii) of Theorem \ref{KS_2003} we have that $M_{\beta}$ is weakly differentiable and
\begin{equation*}
\left\| \nabla M_{\beta} f\right\|_{L^q(\R^d)} \leq C \left\| M_{\beta-1} f\right\|_{L^{q}(\R^d)} \leq C' \left\|  f\right\|_{L^{p^*}(\R^d)} \leq C'' \|\nabla f\|_{L^{1}(\R^d)}.
\end{equation*}
This shows that the map $f \to |\nabla M_{\beta} f|$ is bounded from $W^{1,1}(\R^d)$ to $L^q(\R^d)$ in this case. We are thus left with the following endpoint question, first posed in \cite{CMa}.

\medskip

\noindent {\bf Question 7.} {\it Let $0 \leq \beta < 1$ and $q = d/(d-\beta)$. Is the map $f \to |\nabla M_{\beta} f|$ bounded from $W^{1,1}(\R^d)$ to $L^q(\R^d)$\,$?$ Same question for the uncentered version $\widetilde{M}_{\beta}$.}

\medskip

A complete answer to Question 7 was achieved in dimension $d=1$ for the uncentered fractional maximal operator $\widetilde{M}_{\beta}$ in \cite[Theorem 1]{CMa}. To state this result we need to introduce a generalized version of the concept of variation of a function. For a function $f: \R \to \R$ and $1\leq q < \infty$, we define its $q$-variation as
\begin{equation}\label{Intro_q_var}
\var_q(f) := \sup_{\mc{P}} \left(\sum_{n=1}^{N-1} \frac{|f(x_{n+1}) - f(x_n)|^q}{|x_{n+1} - x_n|^{q-1} }\right)^{1/q},
\end{equation}
where the supremum is taken over all finite partitions $\mc{P} = \{x_1 < x_2 < \ldots < x_N\}$. This is also known as the {\it Riesz $q$-variation} of $f$ (see, for instance, the discussion in \cite{BaLi} for this object and its generalizations). Naturally, when $q=1$, this is the usual total variation of the function. A classical result of F. Riesz (see \cite[Chapter IX \S4, Theorem 7]{N}) states that, if $1 < q < \infty$, then $\var_q(f) < \infty$ if and only if $f$ is absolutely continuous and its derivative $f'$ belongs to $L^q(\R)$. Moreover, in this case, we have that
\begin{equation*}
\|f'\|_{L^q(\R)} = \var_q(f).
\end{equation*}

In \cite[Theorem 1]{CMa} the author and J. Madrid proved the following regularizing effect.

\begin{theorem}[Carneiro and Madrid, 2017 - cf. \cite{CMa}]\label{Thm5.2}
Let $0\leq \beta < 1$ and $q = 1/(1-\beta)$. Let $f:\R \to \R$ be a function of bounded variation such that $\wt{M}_{\beta}f \not\equiv \infty$. Then $\wt{M}_{\beta}f$ is absolutely continuous and its derivative satisfies
\begin{equation}\label{Intro_q_bound1}
\big\|\big(\wt{M}_{\beta}f\big)'\big\|_{L^q(\R)} = \var_q\big(\wt{M}_{\beta}f\big) \leq 8^{1/q}\,\var(f).
\end{equation}
\end{theorem}
The constant $C = 8^{1/q}$ in \eqref{Intro_q_bound1} arises naturally with the methods employed in \cite{CMa} and it is not necessarily sharp (in fact, we have seen that, when $\beta = 0$ and $q=1$, this inequality holds with constant $C=1$). The problem of finding the sharp constant in this inequality is certainly an interesting one. The strategy of \cite{CMa} to prove Theorem \ref{Thm5.2} in the pure fractional case $\beta >0$  is very different from that of the proof of Theorem \ref{Thm_APL}. While in the proof of Theorem \ref{Thm_APL} the essential idea is to prove that  the maximal function does not have any local maxima in the set where it disconnects from the original function, in the fractional case $\beta >0$, the mere notion of the disconnecting set is ill-posed, since one does not necessarily have $\widetilde{M}_{\beta}(f) (x) \geq |f(x)|$ a.e. anymore. To overcome this challenge, the author and Madrid in \cite{CMa} adopt a suitable bootstrapping procedure to bound the $q$-variation of $\wt{M}_{\beta}f$ on certain intervals by the variation of $f$ in larger (but still somewhat comparable) intervals.

\smallskip

In the higher dimensional case, partial progress on Question 7 was obtained by Luiro and Madrid in the recent work \cite{LuMa}. They considered the uncentered fractional maximal operator $\widetilde{M}_{\beta}$ {\it acting on radial functions}. The following result is therefore the fractional analogue of Theorem \ref{Luiro2017}.

\begin{theorem}[Luiro and Madrid, 2017 - cf. \cite{LuMa}]\label{LuMa_Th5.3}
Given $0<\beta <1$ and $q = d/(d-\beta)$, there is a constant $C = C(d,\beta)$ such that for every radial function $f \in W^{1,1}(\R^d)$ we have that $\widetilde{M}_{\beta}f$ is weakly differentiable and 
$$\big\|\nabla \widetilde{M}_{\beta} f\big\|_{L^q(\R^d)} \leq C \,\|\nabla f\|_{L^1(\R^d)}.$$
\end{theorem}
\noindent We can hence ask ourselves the follow up question, which is a particular case of Question 7.

\medskip

\noindent {\bf Question 8.} {\it Does Theorem \ref{LuMa_Th5.3} hold for the centered fractional maximal operator $M_{\beta}$ acting on radial functions $f$}\,?

\medskip

Also in the higher dimensional case, the very interesting recent work of Beltran, Ramos and Saari \cite{BRS} establishes endpoint bounds for derivatives of fractional maximal functions in the spirit of the ones proposed in Question 7. They consider the slightly different setting of maximal operators either associated to smooth convolution kernels or to a lacunary set of radii in dimensions $d \geq 2$ (see \cite[Theorem 1]{BRS}). In this work, they also show that the spherical maximal operator maps $L^p$ into a first order Sobolev spaces in dimensions $d \geq 5$. One of the novelties in the approach of \cite{BRS} is the use of Fourier analysis techniques.

\section{Discrete analogues}\label{Sec6}

The problems we have discussed so far can also be considered in a discrete setup. A point $n \in \Z^d$ is a $d$-uple $n = (n_1, n_2, \ldots, n_d)$ with each $n_i \in \Z$. For a function $f:\mathbb{Z}^{d}\rightarrow \mathbb{R}$ (or, in general, for a vector-valued function $f:\mathbb{Z}^{d}\rightarrow \mathbb{R}^m$) we define its $\ell^{p}$-norm as usual:
\begin{equation}\label{Intro_l_p_norm}
\|f\|_{\ell^{p}{( \Z^{d})}}= \left(\sum_{ n\in \Z^{d}} {|f(n)|^{p}}\right)^{1/p},
\end{equation}
if $1\leq p<\infty$, and
\begin{equation*}
\|f\|_{\ell^{\infty}{(\Z^{d})}}= \sup_{n\in\Z^{d} }{|f(n)|}.
\end{equation*}
The gradient $\nabla{f}$ of a discrete function $f:\Z^d \to \R$ is the vector
\begin{equation*}
\nabla{f(n)}= \big( \partial_1{f} (n), \partial_2{f} (n), \ldots, \partial_d{f} (n)\big),
\end{equation*}
where
\begin{equation*}
\partial_i{f} (n)=f(n+e_{i})-f(n),
\end{equation*}
and $e_{i}=(0,0,\ldots,1,\ldots,0)$ is the canonical $i-$th base vector. If $f: \Z \to \R$ is a given function, we define its total variation as
$$\var(f) = \|f'\|_{\ell^1(\Z)} = \sum_{n \in \Z} |f(n+1) - f(n)|.$$

\smallskip

If $f \in \ell^p(\Z^d)$, observe by the triangle inequality that we have $\nabla f \in \ell^p(\Z^d)$ as well. Therefore, if we were to copy and paste the definition of the Sobolev space $W^{1,p}(\R^d)$ to the discrete setting, we would simply find the space $\ell^p(\Z^d)$ with a norm equivalent to \eqref{Intro_l_p_norm}. Hence, in what follows, some of the questions that were formulated using the Sobolev spaces $W^{1,p}(\R^d)$ in the continuous setting will now be formulated within $\ell^p(\Z^d)$.

\subsection{One-dimensional results} We may start by defining the discrete analogue of \eqref{HL_Max} in the one-dimensional case. For $f: \Z \to \R$ we define the discrete centered one-dimensional Hardy-Littlewood maximal function $\M f:\Z \to \R$ by
$$\M f(n) = \sup_{r\geq0} \frac{1}{(2r+1)}\sum_{k = -r}^{r} |f(n+k)|,$$
where the supremum is taken over nonnegative and integer values of $r$. Analogously, we define the uncentered version of this operator by
$$\widetilde{\M} f(n) = \sup_{r,s\geq0} \frac{1}{(r+s+1)}\sum_{k = -r}^{s} |f(n+k)|,$$
where the supremum is taken over nonnegative and integer values of $r$ and $s$. As in the continuous case, the uncentered version is more friendly for the sort of questions we investigate here. For instance, the analogue of Theorem \ref{Thm_APL} was established in \cite[Theorem 1]{BCHP}.
\begin{theorem}[Bober, Carneiro, Pierce and Hughes, 2012 - cf. \cite{BCHP}] If $f:\Z \to \R$ is a function of bounded variation, then
$$\var\big( \widetilde{\M} f\big) \leq \var(f).$$
\end{theorem}
\noindent This inequality is sharp as one can see by the ``delta" example $f(0) = 1$ and $f(n) = 0$ for $n \neq 0$. The same sort of inequality in the centered case is subtler. Assume for a moment that we have 
\begin{equation}\label{disc_cent}
\var(\M f) \leq \var(f)
\end{equation}
for any $f:\Z \to \R$ of bounded variation. Then, by \eqref{disc_cent} and an application of the triangle inequality, we would have the weaker inequality
\begin{equation}\label{disc_cent_2}
\var( \M f) \leq 2\|f\|_{\ell^1(\Z)}.
\end{equation}
Inequality \eqref{disc_cent_2} was proved in \cite[Theorem 1]{BCHP} with constant $C = 2 + \frac{146}{315}$ replacing the constant $C =2$, and it was proved with the sharp constant $C =2 $ in the recent work of Madrid \cite[Theorem 1.1]{M}. The fact that $C=2$ is sharp in \eqref{disc_cent_2} is again seen by taking the delta example.

\smallskip

The interesting part of the story is that \eqref{disc_cent} is still not known. The BV-boundedness in the discrete centered case was proved by Temur \cite{Te}, adapting the circle of ideas developed by Kurka \cite{Ku} for the continuous case.

\begin{theorem}[Temur, 2013 - cf. \cite{Te}]
If $f:\Z \to \R$ is a function of bounded variation, then
$$\var(\M f) \leq C \,\var(f)$$
with $C = (72000)2^{12} + 4 = 294912004$.
\end{theorem}

We record here the open inequality \eqref{disc_cent}.

\medskip

\noindent {\bf Question 9.} {\it Let $f:\Z \to \R$ be a function of bounded variation. Do we have
$$\var(\M f) \leq  \var(f)\,?$$
Or at least, can one substantially improve on Temur's constant $C = 294912004$}\,?

\medskip

Having discussed the classical case, we may now consider the {\it discrete fractional} case. For $0 \leq \beta < 1$ and $f :\Z\to \R$, we define the one-dimensional discrete centered fractional maximal operator by
\begin{equation*}
\M_{\beta}f(n) = \sup_{r \geq 0} \,\frac{1}{(2r+1)^{1 - \beta}} \sum_{k = -r}^{r} |f(n +k)|
\end{equation*}
and its uncentered version by
\begin{equation*}
\wt{\M}_{\beta}f(n) = \sup_{r,s \geq 0} \,\frac{1}{(r + s+1)^{1 - \beta}} \sum_{k = -r}^{s} |f(n +k)|.
\end{equation*}
For $f :\Z\to \R$ and $1\leq q < \infty$, the discrete analogue of \eqref{Intro_q_var} is the $q$-variation defined by
\begin{equation*}
\var_q(f) = \left(\sum_{n\in \Z} |f(n+1) - f(n)|^q\right)^{1/q} = \|f'\|_{\ell^q(\Z)}.
\end{equation*}
The discrete analogue of Theorem \ref{Thm5.2} was also established in \cite{CMa}.
\begin{theorem}[Carneiro and Madrid, 2017 - cf. \cite{CMa}]
Let $0\leq \beta < 1$ and $q = 1/(1-\beta)$. Let $f:\Z \to \R$ be a function of bounded variation such that $\wt{\M}_{\beta}f \not\equiv \infty$. Then 
\begin{equation*}
\big\|\big(\wt{\M}_{\beta}f\big)'\big\|_{\ell^q(\Z)} = \var_q\big(\wt{\M}_{\beta}f\big) \leq 4^{1/q} \,\var(f).
\end{equation*}
\end{theorem}
\noindent As in the continuous case, we remark that the constant $C = 4^{1/q}$ above is not necessarily sharp. The same inequality for the centered fractional case is currently an open problem.

\medskip

\noindent {\bf Question 10.} {\it Let $0< \beta < 1$ and $q = 1/(1-\beta)$. Let $f:\Z \to \R$ be a function of bounded variation such that $\M_{\beta}f \not\equiv \infty$. Do we have
$$\var_q(\M_{\beta}f) \leq C \,\var(f)$$
for some universal constant $C$?}

\subsection{Multidimensional results} In discussing the multidimensional discrete setting we allow ourselves a more general formulation, in which we consider maximal operators associated to general convex sets. Let $\Omega\subset \mathbb R^{d}$ be a bounded open convex set with Lipschitz boundary. Let us assume that $ 0\in$ int$(\Omega)$. For $r>0$ we write 
\begin{equation*}
\overline\Omega(x,r) =\big\{y \in\mathbb R^{d}; \, r^{-1}(y-x)\in \overline{\Omega}\big\},
\end{equation*}
and for $r=0$ we consider
\begin{equation*}
\overline\Omega(x,0) =\{x\}.
\end{equation*}
This object is the ``$\Omega$-ball of center $x$ and radius $r$" in our maximal operators below. For instance, to work with regular $\ell^{p}-$balls, one should consider $\Omega=\{ x\in\R^{d}; \|x\|_{p}<1 \}$, where $\|x\|_{p}=(|x_{1}|^{p}+|x_{2}|^{p}+\ldots+|x_{d}|^{p})^{\frac{1}{p}} $ for $ x=(x_{1},x_{2},\ldots,x_{d})\in \R^{d}.$ These convex $\Omega$--balls have roughly the same behavior as the regular Euclidean balls from the geometric and arithmetic points of view. For instance, we have the following asymptotics \cite[Chapter VI \S 2, Theorem 2]{Lang}, for the number of lattice points $N(x,r)$ of $\overline\Omega(x,r)$,
\begin{equation*}
 N(x,r)=C_{\Omega}\,r^{d}+O\big(r^{d-1}\big) 
\end{equation*} 
as $r \rightarrow \infty$, where $C_{\Omega}=m(\Omega)$ is the $d$--dimensional volume of $\Omega,$ and the constant implicit in the big O notation depends only on the dimension $d$ and on the set $\Omega$.

\smallskip

Given $0\leq \beta<d$ and $f:\Z^d \to \R$, we denote by $\M_{\Omega, \beta}$ the discrete centered fractional maximal operator associated to $\Omega$, i.e. 
\begin{equation*}
\M_{\Omega, \beta}f(n)=\sup_{r\geq 0} \frac{1}{N(0,r)^{1- \frac{\beta}{d}}}\,\sum_{m\in \overline\Omega(0,r) }{|f(n+m)}|,
\end{equation*}
and we denote by $\widetilde{\M}_{\Omega,\beta}$ its uncentered version
\begin{equation*}
\widetilde{\M}_{\Omega,\beta}f(n)=\sup_{\overline\Omega(x,r) \owns  n}\, \frac{1}{N(x,r)^{1-\frac{\beta}{d}}}\,{\sum_{m\in \overline\Omega (x,r)}{|f(m)}|}.
\end{equation*}
Here $r$ is a nonnegative real parameter. 

\smallskip

The $\ell^p-\ell^q$ boundedness numerology for these discrete operators is the very same as the continuous fractional Hardy-Littlewood maximal operator (see \cite{CMa} for a discussion), that is, if $1 < p < \infty$, $0 < \beta < d/p$ and $q = dp/(d-\beta p)$, then $\M_{\Omega, \beta}: \ell^p(\Z^d) \to \ell^q(\Z^d)$ is bounded (same for $\widetilde{\M}_{\Omega,\beta}$). Motivated by the endpoint philosophy in the continuous setting, a typical question here should be: let $0 \leq \beta <d$ and $q = d/(d-\beta)$; for a discrete function $f:\Z^d \to \R$ do we have $\|\nabla \M_{\Omega,\beta}f\|_{\ell^{q}(\Z^{d})}\leq C (d, \Omega,\beta) \,\|\nabla f\|_{\ell^{1}(\Z^{d})}$? (same question for $\widetilde{\M}_{\Omega,\beta}$). As in the continuous case, this question admits a positive answer if $1 \leq \beta <d$. In the harder case $0\leq \beta < 1$, the current state of affairs is that one has a family of estimates that {\it approximate} the conjectured bounds (but unfortunately blow up when one tries to get exactly there). This was established in \cite[Theorem 3]{CMa} and we quote below.
\begin{theorem}[Carneiro and Madrid, 2017 - cf. \cite{CMa}].\label{Thm6.4}
\begin{enumerate}
\item[(i)] Let $0\leq \beta <d$ and $0 \leq \alpha \leq 1$. Let $q \geq 1$ be such that 
\begin{equation*}
q > \frac{d}{d - \beta + \alpha}.
\end{equation*}
Then there exists a constant $C = C(d,\Omega,\alpha,\beta,q) >0$ such that
\begin{equation}\label{Thm6.4_eq}
\|\nabla \M_{\Omega,\beta}f\|_{\ell^q(\Z^d)} \leq C\,\|\nabla f\|_{\ell^1(\Z^d)}^{1-\alpha}\,\|f\|^{\alpha}_{\ell^{1}(\Z^{d})}\ \ \forall f\in \ell^{1}(\Z^{d}).
\end{equation}
Moreover, the operator $f\mapsto \nabla \M_{\Omega,\beta}f$ is continuous from $\ell^{1}(\Z^d)$ to $\ell^{q}(\Z^{d})$.

\smallskip

\item[(ii)] Let $1\leq \beta <d$ and $0 \leq \alpha < 1$. Let 
\begin{equation*}
q = \frac{d}{d - \beta + \alpha}.
\end{equation*}
Then there exists a constant $C = C(d,\Omega,\alpha,\beta) >0$ such that
\begin{equation*}
\|\nabla \M_{\Omega,\beta}f\|_{\ell^q(\Z^d)} \leq C\,\|\nabla f\|_{\ell^1(\Z^d)}^{1-\alpha}\,\|f\|^{\alpha}_{\ell^{1}(\Z^{d})}\ \ \forall f\in \ell^{1}(\Z^{d}).
\end{equation*}
Moreover, the operator $f\mapsto \nabla \M_{\Omega,\beta}f$ is continuous from $\ell^{1}(\Z^d)$ to $\ell^{q}(\Z^{d})$.
\end{enumerate} 
 
\smallskip

The same results hold for the discrete uncentered fractional maximal operator $\widetilde{\M}_{\Omega,\beta}$. 
\end{theorem}

\noindent {\sc Remark:} Theorem \ref{Thm6.4} already brings some continuity statements. These shall be further discussed in the next section.

\smallskip

By a suitable dilation argument, in \cite{CMa} it is shown that inequality \eqref{Thm6.4_eq} can only hold if 
\begin{equation}\label{Eq_CM}
q \geq \frac{d}{d - \beta + \alpha}.
\end{equation}
The argument to show \eqref{Eq_CM} goes roughly as follows. Consider, for instance, the uncentered case where $\Omega = (-1,1)^d$ is the unit open cube. Let $k \in \N$ and consider the cube $Q_k = [-k,k]^d$ and its characteristic function $f_k := \chi_{Q_k}$. One has $\|f_k\|_{\ell^{1}(\Z^d)} \sim_d k^d$, $\|\nabla f_k\|_{\ell^1(\Z^d)} \sim_d k^{d-1}$ and $\|\nabla \wt{\M}_{\Omega,\beta}f_k\|_{\ell^q(\Z^d)} \gg_{\Omega,\beta,d} k^{\frac{d}{q}-1 + \beta}$. One can see this last estimate by considering the region $H = \{n = (n_1, n_2, \ldots, n_d) \in \Z^d; \,n_1 \geq 4dk\,; \ |n_i| \leq k,\, {\rm for}\,\, i = 2,3,\ldots, d\}$ and showing that the maximal function at $n \in H$ is realized by the cube of side $n_1 + k$ that contains the cube $Q_k$. Then we sum $|\wt{\M}_{\Omega,\beta}f_k (n + e_1) - \wt{\M}_{\Omega,\beta}f_k (n)|^q$ from $n_1 = 4dk$ to $\infty$, and then sum these contributions over the $\sim k^{d-1}$ possibilities for $(n_2, \ldots, n_d)$. Letting $k \to \infty$ we obtain the necessary condition \eqref{Eq_CM}. 

\smallskip

This leaves us the following open question.

\medskip

\noindent {\bf Question 11.} Let $0 \leq \beta <1$ and $q = d/(d-\beta)$. For a discrete function $f:\Z^d \to \R$ do we have $\|\nabla \M_{\Omega,\beta}f\|_{\ell^{q}(\Z^{d})}\leq C (d, \Omega,\beta) \,\|\nabla f\|_{\ell^{1}(\Z^{d})}$? More generally, does the inequality \eqref{Thm6.4_eq} hold for all $\alpha \leq \beta$ and $q = d/(d - \beta + \alpha)$? (Analogous questions for $\widetilde{\M}_{\Omega,\beta}$).

\section{Continuity}

We now turn to the final chapter of our discussion, in which we consider the continuity properties of the mappings we have addressed so far. The classical Hardy-Littlewood maximal operator $M$ is a sublinear operator, i.e. $M(f + g)(x) \leq Mf(x) + Mg(x)$ pointwise. Having this property at hand, it is easy to see that the fact that $M:L^p(\R^d) \to L^p(\R^d)$ is bounded (for $1 < p \leq \infty$) implies that this map is also continuous. In fact, if $f_j \to f$ in $L^p(\R^d)$, then
$$\|Mf_j - Mf\|_{L^p(\R^d)} \leq \|M(f_j - f)\|_{L^p(\R^d)} \leq C\,\|f_j - f\|_{L^p(\R^d)} \to 0.$$
Same reasoning applies to its uncentered, fractional or discrete versions (all being sublinear operators).

\smallskip

At the level of the (weak) derivatives, these operators, in principle, are not necessarily sublinear anymore. In light of the boundedness of the operator $M: W^{1,p}(\R^d) \to W^{1,p}(\R^d)$, for $1 < p \leq \infty$, established by Kinnunen, it is then a natural and nontrivial question to ask whether this operator is also continuous. This question is attributed to T. Iwaniec and was first explicitly posed in the work of Haj\l asz and Onninen \cite[Question 3]{HO}, in the case $1<p<\infty$. It was settled affirmatively by Luiro in \cite[Theorem 4.1]{Lu1}.
\begin{theorem}[Luiro, 2007 - cf. \cite{Lu1}]\label{Thm_cont_Luiro}
The operator $M: W^{1,p}(\R^d) \to W^{1,p}(\R^d)$ is continuous for $1 < p < \infty$.
\end{theorem}

\noindent{\sc Remark:} In the case $p= \infty$, the continuity of $M: W^{1,\infty}(\R^d) \to W^{1,\infty}(\R^d)$ does not hold, as pointed out to the author by H. Luiro. A counterexample may be constructed along the following lines (in dimension $d=1$, say). Take a smooth $f$ with compact support such that $(Mf)'$ has a point of discontinuity. Letting $f_h(x) = f(x + h)$, one sees that $f_h \to f$ in $W^{1,\infty}(\R)$ as $h \to 0$, but $(M(f_h))' =(Mf)'_h \nrightarrow (Mf)'$ in $L^{\infty}(\R)$ as $h \to 0$.

\smallskip

The proof of Luiro for Theorem \ref{Thm_cont_Luiro} is very elegant. It provides an important qualitative study of the convergence properties of the sets of ``good radii" (i.e. the radii that realize the supremum in the definition of the maximal function) and establishes an explicit formula for the derivative of the maximal function (in which one is able to move the derivative inside the integral over a ball of good radius). It also uses crucially the $L^p$-boundedness of the maximal operator. A similar study of the continuity properties of the local maximal operator on subdomains of $\R^d$ was also carried out by Luiro in \cite{Lu2}.

\smallskip

\subsection{Endpoint study} As we have done many times before in this paper, we now turn our attention to the endpoint $p=1$. So far, we have established several boundedness results at $p=1$, and we now want to ask ourselves if such maps are continuous. For instance, the very first one of such boundedness results is Tanaka's inequality \eqref{Tanaka_boy} that establishes that $f \mapsto (\widetilde{M}f)'$ is bounded from $W^{1,1}(\R)$ to $L^1(\R)$. The corresponding continuity question is: if $f_j \to f \in W^{1,1}(\R)$ as $j \to \infty$, do we have $(\widetilde{M}f_j)' \to (\widetilde{M}f)'$ in $L^1(\R)$ as $j \to \infty$? This was settled affirmatively in \cite[Theorem 1]{CMP}.
\begin{theorem}[Carneiro, Madrid and Pierce, 2017 - cf. \cite{CMP}]\label{CMP_Thm1}
The map $f \mapsto \big(\widetilde{M}f\big)'$ is continuous from $W^{1,1}(\R)$ to $L^1(\R)$.
\end{theorem}
The proof of this result is quite subtle, and different from Luiro's approach to Theorem \ref{Thm_cont_Luiro} since one does not have the $L^1$-boundedness of the maximal operator. The authors in \cite{CMP} develop a fine analysis towards the required convergence using the qualitative description of the uncentered maximal function (and the one-sided maximal functions) on the disconnecting set. The corresponding question for the one-dimensional centered maximal function $M$ is even more challenging and it is currently open.

\medskip

\noindent {\bf Question 12.} Is the map $f \mapsto (Mf)'$ is continuous from $W^{1,1}(\R)$ to $L^1(\R)$?

\medskip

In light of inequalities \eqref{Var_Sharp} and \eqref{Var_Sharp_centered}, one may ask similar (and harder) continuity questions on the Banach space of normalized functions of bounded variation. Throughout the rest of this section let us denote by $BV(\R)$ the (Banach) space of functions $f:\R \to \R$ of bounded total variation with norm
$$\|f\|_{BV(\R)} = |f(-\infty)| + \var(f),$$
where $f(-\infty) = \lim_{x\to -\infty}f(x)$. Since 
$$\|\widetilde{M}f\|_{L^{\infty}(\R)} \leq \|f\|_{L^{\infty}(\R)} \leq \|f\|_{BV(\R)},$$
together with \eqref{Var_Sharp} we see that $\widetilde{M}: BV(\R) \to BV(\R)$ is bounded. The same holds for $M: BV(\R) \to BV(\R)$. The corresponding continuity statements arise as interesting open problems that would be qualitatively stronger then Theorem \ref{CMP_Thm1} or Question 12, if confirmed.

\medskip

\noindent {\bf Question 13.} Is the map $\widetilde{M}: BV(\R) \to BV(\R)$ continuous?

\medskip

\noindent {\bf Question 14.} Is the map $M: BV(\R) \to BV(\R)$ continuous?

\medskip

\subsection{Fractional setting} We now move the endpoint discussion to the fractional setting as considered in Section \ref{Sec5}. As in the classical setting considered above, we may think of the endpoint continuity questions assuming a (stronger) $W^{1,1}(\R)$ convergence or a (weaker) $BV(\R)$ convergence on the source space. With respect to the first type, the corresponding continuity statement to Theorem \ref{Thm5.2} was established in \cite{M2}.
\begin{theorem}[Madrid, 2018 - cf. \cite{M2}]\label{Thm7.3}
Let $0< \beta < 1$ and $q = 1/(1-\beta)$. The map $f \mapsto \big(\wt{M}_{\beta}f\big)'$ is continuous from $W^{1,1}(\R)$ to $L^q(\R)$.
\end{theorem}
The analogous continuity question for the centered one-dimensional fractional maximal operator, for which the boundedness is not yet known (see Question 7 above), is an interesting open problem.

\medskip

\noindent {\bf Question 15.} Let $0< \beta < 1$ and $q = 1/(1-\beta)$. Is the map $f \mapsto (M_{\beta}f)'$ bounded and continuous from $W^{1,1}(\R)$ to $L^q(\R)$?

\medskip

With respect to the second type of continuity statement, in which one assumes the $BV(\R)$-convergence on the source space, the interesting fact is that the fractional endpoint maps {\it are not continuous.} This was shown in \cite[Theorems 3 and 4]{CMP} and we quote the results below.

\begin{theorem}[Carneiro, Madrid and Pierce, 2017 - cf. \cite{CMP}]\label{Thm7.4}
Let $0< \beta < 1$ and $q = 1/(1-\beta)$.
\begin{itemize}
\item[(i)] (uncentered case) The map $f \mapsto \big(\wt{M}_{\beta}f\big)'$ is not continuous from $BV(\R)$ to $L^q(\R)$, i.e. there is a sequence $\{f_j\}_{j\geq 1} \subset BV(\R)$ and a function $f \in BV(\R)$ such that $\|f_j - f\|_{BV(\R)} \to 0$ as $j \to \infty$ but 
$$\big\|\big(\wt{M}_{\beta}f_j\big)' - \big(\wt{M}_{\beta}f\big)'\big\|_{L^q(\R)} = \var_q \big(\wt{M}_{\beta}f_j - \wt{M}_{\beta}f\big) \nrightarrow 0$$ 
as $j \to \infty$.

\smallskip

\item[(ii)] (centered case) There is a sequence $\{f_j\}_{j\geq 1} \subset BV(\R)$ and a function $f \in BV(\R)$ such that $\|f_j - f\|_{BV(\R)} \to 0$ as $j \to \infty$ but $\var_q \big(M_{\beta}f_j - M_{\beta}f\big) \nrightarrow 0$ as $j \to \infty$.
\end{itemize}
\end{theorem}
Notice the slightly different wording in the items (i) and (ii) of the theorem above. The reason is that in the centered case we do not know yet if the analogue of Theorem \ref{Thm5.2} holds. Theorem \ref{Thm7.4} (ii) says that, regardless of the map $f \mapsto (M_{\beta}f)'$ being bounded from $BV(\R)$ to $L^q(\R)$ or not, it is not continuous.

\subsection{Discrete setting} To consider similar continuity issues in the discrete setting we define the Banach space $BV(\Z)$ as the space of functions $f:\Z \to \R$ of bounded total variation with norm
\begin{equation*}
\|f\|_{{\rm BV(\Z)}} = \big|f(-\infty)\big| + \var(f),
\end{equation*}
where $f(-\infty):= \lim_{n\to -\infty} f(n)$. Recall the discussion on the beginning of Section \ref{Sec6} in which we said that there is no actual space $W^{1,1}(\Z)$, as this is simple $\ell^1(\Z)$ with a different norm. Then, in the instances where we assumed a $W^{1,1}(\R)$-convergence in the continuous setting, we will be assuming an $\ell^1(\Z)$-convergence in the discrete setting. As a particular case of the general framework of Theorem \ref{Thm6.4} above (which is \cite[Theorem 3]{CMa}) we see that the maps $f \mapsto (\M_{\beta}f)'$ and $f \mapsto (\wt{\M}_{\beta}f)'$ are continuous from $\ell^1(\Z)$ to $\ell^q(\Z)$ for $0 \leq \beta <1$ and $q = 1/(1-\beta)$ (the case $\beta = 0$ of these results had previously been obtained in \cite{CH}). Therefore, we have an affirmative answer for the discrete analogues of Theorems \ref{CMP_Thm1} and \ref{Thm7.3} and Questions 12 and 15.

\smallskip

The $BV(\Z)$-continuity is a much more interesting issue. For the classical discrete Hardy-Littlewood maximal operators, we can affirmatively answer the analogues of Questions 13 and 14. This was accomplished in \cite[Theorem 2]{CMP} for the uncentered case and in \cite[Theorem 1.2]{M2} for the centered case. We collect these results below.

\begin{theorem}[Carneiro, Madrid and Pierce, 2017 - cf. \cite{CMP}] \label{Thm7.5}
The map $\wt{\M} :BV(\Z) \to BV(\Z)$ is continuous.
\end{theorem}

\begin{theorem}[Madrid, 2018 - cf. \cite{M2}]\label{Thm7.6}
The map $\M :BV(\Z) \to BV(\Z)$ is continuous.
\end{theorem}

As in the continuous cases, the fractional discrete maximal operators are not continuous on $BV(\Z)$, as observed in \cite[Theorems 5 and 6]{CMP}.

\begin{theorem}[Carneiro, Madrid and Pierce, 2017 - cf. \cite{CMP}]\label{Thm7.7}
Let $0< \beta < 1$ and $q = 1/(1-\beta)$.
\begin{itemize}
\item[(i)] (uncentered case) The map $f \mapsto \big(\wt{\M}_{\beta}f\big)'$ is not continuous from $BV(\Z)$ to $\ell^q(\Z)$, i.e. there is a sequence $\{f_j\}_{j\geq 1} \subset BV(\Z)$ and a function $f \in BV(\Z)$ such that $\|f_j - f\|_{BV(\Z)} \to 0$ as $j \to \infty$ but 
$$\big\|\big(\wt{\M}_{\beta}f_j\big)' - \big(\wt{\M}_{\beta}f\big)'\big\|_{\ell^q(\Z)} = \var_q \big(\wt{\M}_{\beta}f_j - \wt{\M}_{\beta}f\big) \nrightarrow 0$$ 
as $j \to \infty$.

\smallskip

\item[(ii)] (centered case) There is a sequence $\{f_j\}_{j\geq 1} \subset BV(\Z)$ and a function $f \in BV(\Z)$ such that $\|f_j - f\|_{BV(\Z)} \to 0$ as $j \to \infty$ but $\var_q \big(\M_{\beta}f_j - \M_{\beta}f\big) \nrightarrow 0$ as $j \to \infty$.
\end{itemize}
\end{theorem}
Note again the slight difference in the wording between parts (i) and (ii) of the statement above. This is due to the fact that the map $f \mapsto (\M_{\beta}f)'$ is not yet known to be bounded from $BV(\Z)$ to $\ell^q(\Z)$. Nevertheless, it is not continuous.

\subsection{Summary} The table below collects the sixteen different situations in which we analyzed the endpoint continuity (all of them one-dimensional). These arise from the following pairs of possibilities: (i) centered vs. uncentered maximal operator; (ii) classical vs. fractional maximal operator; (iii) continuous vs. discrete setting; (iv) $W^{1,1}$ (or $\ell^1$) vs. $BV$ continuity. The word YES in a box below means that the continuity of the corresponding map has been established, whereas the word NO means that the continuity fails. The remaining boxes are marked as OPEN problems.

\begin{table}[h]
\renewcommand{\arraystretch}{1.3}
\centering
\caption{One-dimensional endpoint continuity program}
\label{Table-ECP}
\begin{tabular}{|c|c|c|c|c|}
\hline
 \raisebox{-1.3\height}{------------}&  \parbox[t]{2.8cm}{  $W^{1,1}-$continuity; \\ continuous setting} &  \parbox[t]{2.8cm}{  $BV-$continuity; \\ continuous setting} & \parbox[t]{2.6cm}{  $\ell^1-$continuity; \\ discrete setting} &  \parbox[t]{2.5cm}{  $BV-$continuity; \\ discrete setting} \\ [0.5cm]
\hline
 \parbox[t]{3.3cm}{ Centered classical \\ maximal operator} &  \raisebox{-0.8\height}{OPEN: Question 12} & \raisebox{-0.8\height}{OPEN: Question 14} & \raisebox{-0.8\height}{YES$^2$: Thm \ref{Thm6.4}}  &\raisebox{-0.8\height}{YES: Thm \ref{Thm7.6}}\\[0.5cm]
 \hline
  \parbox[t]{3.3cm}{ Uncentered classical \\ maximal operator} &  \raisebox{-0.8\height}{YES: Thm 7.2} &  \raisebox{-0.8\height}{OPEN: Question 13} &  \raisebox{-0.8\height}{YES$^2$: Thm \ref{Thm6.4}} &  \raisebox{-0.8\height}{YES: Thm \ref{Thm7.5}} \\[0.5cm]
 \hline
 \parbox[t]{3.3cm}{ Centered fractional \\ maximal operator} &  \raisebox{-0.8\height}{OPEN$^1$: Question 15} &  \raisebox{-0.8\height}{NO$^1$: Thm \ref{Thm7.4}} &  \raisebox{-0.8\height}{YES: Thm \ref{Thm6.4} } & \raisebox{-0.8\height}{NO$^1$: Thm \ref{Thm7.7}}  \\[0.5cm]
 \hline
\parbox[t]{3.3cm}{ Uncentered fractional \\ maximal operator} &  \raisebox{-0.8\height}{YES: Thm 7.3} &  \raisebox{-0.8\height}{NO: Thm \ref{Thm7.4}} &  \raisebox{-0.8\height}{YES: Thm \ref{Thm6.4}}  & \raisebox{-0.8\height}{NO: Thm \ref{Thm7.7}} \\[0.5cm]
 \hline
\end{tabular}
\vspace{0.05cm}
\flushleft{
\ \ \footnotesize{$^1$ Corresponding boundedness result not yet known.}\\
\ \ $^2$ Result previously obtained in \cite[Theorem 1]{CH}.}
\end{table}

\section{Acknowledgments}
\noindent The author acknowledges support from CNPq-Brazil grants $305612/2014-0$ and $477218/2013-0$, and FAPERJ grant $E-26/103.010/2012$. The author is thankful to Jos\'{e} Madrid and Jo\~{a}o Pedro Ramos for reviewing the manuscript and providing valuable feedback. The author is also thankful to all of the members of the Scientific and Organizing Committees of the {\it CIMPA 2017 Research School - Harmonic Analysis, Geometric Measure Theory and Applications}, in Buenos Aires, Argentina, for the wonderful event.

\end{document}